\documentclass[oneside,english,a4paper]{amsart}
\usepackage[T1]{fontenc}
\usepackage[a4paper]{geometry}
\usepackage{amstext}
\usepackage{amsthm}
\usepackage{amssymb}

\makeatletter
\numberwithin{equation}{section}
\numberwithin{figure}{section}

\newtheorem{theorem}{Theorem}[section]
 \newtheorem{thm}[theorem]{Theorem}
 \newtheorem{assumption}{Assumption}
 \newtheorem{rem}[theorem]{Remark}
 \newtheorem{defn}[theorem]{Definition}
 \newtheorem{lem}[theorem]{Lemma}
 \newtheorem{prop}[theorem]{Proposition}
 \newtheorem{cor}[theorem]{Corollary}
 
 \newtheorem{ex}[theorem]{Example}

\usepackage[sort&compress, numbers]{natbib}

\newcommand{\be}{\begin{equation}}
\newcommand{\ee}{\end{equation}}
\newcommand{\bea}{\begin{eqnarray}}
\newcommand{\eea}{\end{eqnarray}}
\newcommand{\beann}{\begin{eqnarray*}}
\newcommand{\eeann}{\end{eqnarray*}}
\newcommand{\benn}{\begin{equation*}}
\newcommand{\eenn}{\end{equation*}}

\def\ra{\rightarrow}

\def\txtd{{\textnormal{d}}}
\def\txte{{\textnormal{e}}}
\def\txti{{\textnormal{i}}}

\def\txtD{{\textnormal{D}}}

\usepackage{graphicx}

\usepackage{overpic,color}

\newcommand{\eps}{\varepsilon}
\newcommand{\vp}{\varphi}

\numberwithin{equation}{section}

\newcommand{\Bcal}{{\mathcal B}}

\newcommand{\Fcal}{{\mathcal F}}

\newcommand{\Xcal}{{\mathcal X}}

\newcommand{\R}{\mathbb{R}}
\newcommand{\N}{\mathbb{N}}

\renewcommand{\P}{\mathbb{P}}
\newcommand{\E}{\mathbb{E}}

\newcommand{\C}{\mathbb{C}}

\subjclass[2010]{Primary: 60H15; 92C20; Secondary: 34B10; 35B65;}

\AtBeginDocument{
  
}

\makeatother

\usepackage{babel}
\begin{document}
\title[Gradient flows for the stochastic Amari neural field model]{A gradient flow formulation for the stochastic Amari neural field
model}
\author{Christian Kuehn}
\address{Technical University of Munich\\
Faculty of Mathematics\\
Research Unit ``Multiscale and Stochastic Dynamics''\\
85748 Garching bei M\"{u}nchen\\
Germany}
\email{ckuehn@ma.tum.de}
\author{Jonas M. T\"{o}lle}
\address{Universit\"{a}t Augsburg\\
Institut f\"{u}r Mathematik\\
86135 Augsburg\\
Germany}
\curraddr{Universit\"{a}t Ulm\\
Fakult\"{a}t f\"{u}r Mathematik und Wirtschaftswissenschaften\\
Institut f\"{u}r Analysis\\
89069 Ulm\\
Germany}
\email{jonas.toelle@math.uni-augsburg.de}
\begin{abstract}
We study stochastic Amari-type neural field equations, which are mean-field
models for neural activity in the cortex. We prove that under certain
assumptions on the coupling kernel, the neural field model can be
viewed as a gradient flow in a nonlocal Hilbert space. This makes
all gradient flow methods available for the analysis, which could
previously not be used, as it was not known, whether a rigorous gradient
flow formulation exists. We show that the equation is well-posed in
the nonlocal Hilbert space in the sense that solutions starting in
this space also remain in it for all times and space-time regularity
results hold for the case of spatially correlated noise. Uniqueness
of invariant measures, ergodic properties for the associated Feller
semigroups, and several examples of kernels are also discussed. 
\end{abstract}

\date{\today}
\keywords{Gradient flow in Hilbert space; stochastic Amari neural field equation;
nonlocal Hilbert space; spatially correlated additive
noise; space-time regularity of solutions; nonnegative kernel; unique
invariant measure; ergodic Feller semigroup; mathematical neuroscience. }
\thanks{CK acknowledges financial support by a Lichtenberg Professorship.
JMT would like to thank Dirk Bl\"{o}mker for some useful comments.
Both authors are indebted to the anonymous reviewers for several helpful
remarks.}
\maketitle

\section{Introduction}

The stochastic integro-differential equation we study is \be
\label{eq:introAmari} \partial_t u =-\alpha u + \int_{\Bcal} w(\cdot,y)f(u(y,t))\,
\txtd y+\varepsilon\partial_t W \ee where $(x,t)\in\Bcal\times[0,T)$,
$\Bcal\subset\R^{d}$ is the spatial domain modeling the cortex, $u=u(x,t)\in\R$
is the neural field interpreted as the average/mean-field activity
of the neural network, the kernel $w$ models the connections between
neurons, $f$ is the main nonlinearity representing the firing rate,
$\alpha>0$ is the decay rate, and $W=W(x,t)$ is a (spatially correlated,
additive) noise controlled via the parameter $\varepsilon\geq0$.
We shall write $BW$ instead of $W$ below, where $B$ is the covariance
operator. The precise mathematical formulation for~(\ref{eq:introAmari})
starts in the subsequent Section~\ref{sec:The-stochastic-Amari}.
In this introduction, we outline the main setting and our results.

Deterministic neural field models are well-established in the modeling
of mean-field cortex activity. The version~(\ref{eq:introAmari})
(for $\varepsilon=0$) is frequently attributed to \citet{Amari}
while a version with $f$ outside the integral is attributed to \citet{WilsonCowan}.
Since the kernel $w$ models neuronal synaptic connections, it is
often parametrized to ensure that different connection structures
can be modeled, which may arise e.g.~due to synaptic plasticity.
The noise term may either be viewed as modeling external input noise
or it can be viewed as a correction due to finite population size
or population heterogeneity.

A key motivation to study neural fields arises from neurological disorders,
such as Parkinson or epilepsy. For example, epileptic seizures can
be tracked by recording EEG (or ECoG) data, which essentially represent
a partial neural field measurement on the upper (or lower) layers
of the cortex. The general hypothesis is that there are temporary
phases, also called epileptic seizures, during which the brain operates
in a completely different dynamical regime. Whether this changing
activity pattern can be related/modeled to changing the parameters
in the neural connectivity network, and therefore changing the kernel
in a neural field model, remains an open problem currently under discussion
in multiple scientific communities.

However, even during normal functioning of the brain, several topics
such as wave propagation, visual hallucinations, orientation and memory
functions have been connected to continuum neural field models~\citep{Bressloff}
such as \eqref{eq:introAmari}; see also~\citep{Coombes,Ermentrout3}
for broader introductions and further references. Although the analysis
of spatio-temporal pattern-formation for deterministic neural fields
is well-established, only very recently there has been a surging interest
in stochastic neural field equations. One broad motivation is to understand
the effect of finite population size~\citep{Bressloff3,RiedlerBuckwar,TouboulHermannFaugeras},
while another is to take into account environmental fluctuations.
A very concrete application is to understand the role of noise-induced
fluctuations~\citep{SchwalgerDegerGerstner,WebberBressloff} in perceptual
bistability~\citep{Moreno-BoteRinzelRubin,vanEe}. More precisely,
it has been hypothesized that noise-induced bistability in neural
fields switching between different invariant sets is related to switches
between visual perception of ambiguous visual patterns. An example
that noise-induced switching between locally stable deterministic
steady states is indeed possible for~\eqref{eq:introAmari} is shown
in Figure~\ref{fig:01}.\\

\begin{figure}[htbp] 	\centering 	\begin{overpic}[width=0.95\textwidth]{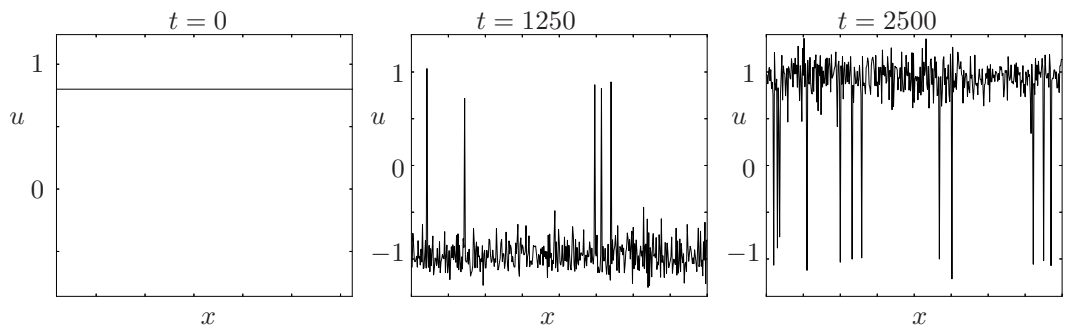}				 		\put(50,0){\scalebox{1}{$x$}} 		\put(85,0){\scalebox{1}{$x$}} 		\put(18,0){\scalebox{1}{$x$}} 		\put(0,19){\scalebox{1}{$u$}} 		\put(34,19){\scalebox{1}{$u$}} 		\put(68,19){\scalebox{1}{$u$}} 		\put(2,12){\scalebox{1}{$0$}} 		\put(2,24){\scalebox{1}{$1$}} 		\put(36,23){\scalebox{1}{$1$}} 		\put(36,14){\scalebox{1}{$0$}} 		\put(34,6){\scalebox{1}{$-1$}} 		\put(69,23){\scalebox{1}{$1$}} 		\put(69,14){\scalebox{1}{$0$}} 		\put(67.5,6){\scalebox{1}{$-1$}} 		\put(15,28){\scalebox{1}{$t=0$}} 		\put(44,28){\scalebox{1}{$t=1250$}} 		\put(78,28){\scalebox{1}{$t=2500$}} 	\end{overpic} 	\caption{\label{fig:01}Numerical simulation~\eqref{eq:introAmari} with 	$f(u)= (u+1)(1-u)(u-0.1)$, $w(x,y)=J(x-y)$ for  	$J(z)=1/(\tilde{\sigma}\sqrt{2\pi})\exp(-z^2/(2\tilde{\sigma}))$ for  	$\tilde{\sigma}=0.05$, $x\in\Bcal=[-80,80]$, $\alpha=0.1$, space-time  	white noise with $\varepsilon=0.5$, and $u(x,0)\equiv 0.8$. The three  	snapshots clearly show that there is metastability between two locally 	deterministically stable states. In particular, the neural field  	behaves reminiscent of a gradient-type system with additive noise.} \end{figure} 

In summary, the main biological insights to be gained from continuum
neural field models are obtained via the analysis of the patterns
they produce, and how these patterns depend upon parameters. A general
observation for many different choices of kernels $w$ and firing
rates $f$ is that multiple patterns are possible, while for certain
parameter ranges the model seems to select just a single pattern.
To clarify this effect, one natural conjecture is that one should
use mathematical analysis to determine, whether there are conditions,
e.g.~on the kernel $w$ that lead to different mathematical evolution
equation structures, which naturally distinguish between a neural
field operating in a multistable pattern formation mode, or in a single
global dissipative mode.

The basic theory for~(\ref{eq:introAmari}), including existence,
regularity, Galerkin approximations and large deviations of~(\ref{eq:introAmari}),
has been covered in the work by~\citet{KuehnRiedler14} for bounded
domains, while unbounded domains are considered by~\citet{FaugerasInglis}.
The influence of noise on traveling waves in stochastic neural fields
has been studied intensively in recent years~\citep{BressloffWebber,InglisMacLaurin,KilpatrickErmentrout,Lang,PollKilpatrick,
KruegerStannat,Kruger:2017eb}. Yet, a sharp study of noise-induced switching rates can only be
carried out for~(\ref{eq:introAmari}) beyond standard large deviations
if one can use a gradient structure, which is required, e.g., to generalize
Kramers' law~\citep{Berglund3}. In fact, trying to transfer the Kramers' law results from the
local stochastic partial differential equation setting in~\citep{Berglund3}
to the nonlocal stochastic integral equation~(\ref{eq:introAmari})
was one of the motivations for this work; see also Figure~\ref{fig:01}
for the numerical confirmation that gradient-like noise-induced switching
dynamics can appear in neural field models.

However, even in the deterministic setting ($\varepsilon=0$) it is
not known, whether we can find a gradient structure to rewrite~(\ref{eq:introAmari})
as \be \label{eq:intrograd} \partial_t u = -\nabla_{\Xcal} \Fcal(u),
\qquad{}u(\cdot,t)=u(t)\in\Xcal, \ee where $\Xcal$ is a suitable
function space, e.g., a Hilbert, Banach or even metric space~\citep{AmbrosioGigliSavare},
and $\Fcal:\Xcal\ra\mathbb{R}$ is a functional, which often has the
natural interpretation of an energy, entropy, or some other physical
notion. Of course, to have such a gradient structure in the (stochastic)
neural field case would not only be interesting for Kramers' law as
suggested by~\citet{KuehnRiedler14} but also open up a general area
of techniques, which has been extremely successful for other differential
equations~\citep{Jordan:1998bf,Otto}.

Furthermore, if we can characterize the types of kernels for which
gradient structures exist, it would give us an understanding, if and
when the brain might be working in two different regimes such as energy-decay
versus complex non-equilibrium pattern formation. Unfortunately, direct
calculations in several works have shown~\citep{EnculescuBestehorn,KuehnRiedler14}
that one does not expect any gradient-structure of the form~(\ref{eq:intrograd})
in the \emph{standard Lebesgue and Sobolev spaces} for neural fields,
despite many analogies to certain classical PDEs with such structures~\citep{LaingTroy}.
See \citep{daSilvaPereira2018} for a recent contribution on the topic
of gradient flows for neural field equations with weights in spaces
of continuous functions.

In this work, we show that a gradient structure actually does exist
for (stochastic) neural fields. We construct a nonlocal Hilbert space
$\Xcal$ using functional-analytic methods, which allows us to obtain
an exact gradient system. For the gradient framework, we analyze well-posedness
of the resulting system in the nonlocal space $\Xcal$ using methods
from stochastic analysis. Using the gradient structure we can characterize
invariant measures of the process. Furthermore, we investigate ergodic
properties of Feller semigroups generated by the neural field equation.
We note that a related idea was used for the stochastic porous media
equation using monotonicity methods by \citet{Ren:2007dm,Rockner:2008cj}.

Furthermore, we crucially note that $\Xcal$, and the related gradient
structure, only exist for a certain classes of kernels with sufficiently
dominant excitation or inhibition (see Section \ref{sec:Examples}
below for examples), while it fails for other classes. Kernels may,
depending upon parameters, fall into different classes so that in
underlying applications we might obtain major transitions in neural
fields between different mathematical structures. In fact, a similar
transition between classes has recently been observed in the context
of entropy/gradient-structures for local cross-diffusion systems~\citep{JuengelKuehnTrussardi}
modeling herding behavior, where parameter variation can destroy the
entropy structure.

We conjecture that our new approach may allow us to now consider many
other stochastic variants of models from biological applications,
e.g., nonlocal Fisher-KPP equations~\citep{AchleitnerKuehn,
BerestyckiNadinPerthameRyzhik,Gourley}, nonlocal aggregation models~\citep{TopazBertozziLewis}, or nonlocal
swarming systems~\citep{MogilnerEdelsteinKeshet} via a gradient
flow approach. Another important future perspective of our framework
is to obtain sharp asymptotics for nonlocal small noise asymptotics
and large deviation principles~\citep[Chapter 12]{DaPrZa:2nd} similar
to Kramers' law recently obtained for stochastic partial differential
equations~\citep{BerglundGentz10,Barret,BerglundDiGesuWeber}.

\subsection*{Structure of the the paper}

In Section \ref{sec:The-stochastic-Amari}, we introduce the mathematical
model for equation \eqref{eq:introAmari} and give the main Assumptions
\ref{assu:one}--\ref{assu:HS-op} for the existence and uniqueness
of solutions to the stochastic evolution equation \eqref{eq:spde-nongradient}.
In Subsection \ref{subsec:Gradients-in}, we shall introduce the nonlocal
Hilbert space and the gradient flow structure on it. The stochastic
gradient flow is introduced in Subsection \ref{subsec:Gradient-flow-formulation},
where also our final Assumption \ref{assu:5} is formulated. Invariance
and pathwise regularity of the solutions are discussed in Subsections
\ref{subsec:Galerkin-approximation} and \ref{subsec:Pathwise-regularity},
respectively. The existence and uniqueness of invariant measures and
the ergodicity of the semigroup associated to the stochastic equation
are studied in Section \ref{sec:Ergodicity}. In Section \ref{sec:Examples},
we give examples for the nonnegativity condition of the kernels and
discuss an equivalent condition for this which is given by Fourier
transforms. The concluding summary is given in Section \ref{sec:Summary-Discussion},
followed by Appendix \ref{sec:Cylindrical-Wiener-processes} on cylindrical
Wiener processes.

\section{\label{sec:The-stochastic-Amari} The stochastic Amari model}

Let $\Bcal\subset\R^{d}$ be non-empty, bounded and closed. From now
on, denote 
\[
H:=L^{2}(\Bcal).
\]

Formalizing equation (\ref{eq:introAmari}), let $\alpha>0$ (decay
rate for the potential), $\varepsilon\ge0$ (noise level parameter)
and consider the following stochastic integro-differential equation
in $H$: 
\begin{equation}
\txtd U_{t}=[-\alpha U_{t}+KF(U_{t})]\,\txtd t+\eps B\,\txtd W_{t},\quad U_{0}=u_{0}\in L^{2}(\Bcal),\quad t\in[0,T],\label{eq:spde-nongradient}
\end{equation}
where $T>0$ is some finite time horizon. Let us discuss the assumptions
on the coefficients.

\begin{assumption} \label{assu:one} Let $f:\R\to(0,+\infty)$ be
a globally Lipschitz continuous \emph{gain function}. \end{assumption}

Typically, we might want to include the \emph{sigmoid function} 
\[
f(s)=(1+\txte^{-s})^{-1}\quad\text{or}\quad f(s)=(\tanh(s)+1)/2
\]
as examples, which are commonly used in the analysis of neural fields~\citep{Bressloff}.
Let $F(g)(x):=f(g(x))$, $g\in H$, be the associated Nemytskii operator
for $f$. Since $f$ is Lipschitz, we get that $F$ is a nonlinear
Lipschitz continuous operator $F:H\to H$, see e.g.~\citep[Theorem II.3.2]{Show}.

Let
\[
Kg(x):=\int_{\Bcal}J(x-y)g(y)\,\txtd y,\quad g\in L^{2}(\Bcal),
\]
where $J:\R^{d}\to\R$, measurable, satisfies the following.

\begin{assumption} \label{assu:new}Assume that
\begin{enumerate}
\item $J(x)=J(-x)$, for every $x\in\Bcal+\Bcal:=\{x+y\;:\;x,y\in\Bcal\}$,
\item $J\in L^{2}(\Bcal+\Bcal)\cap C(\Bcal+\Bcal)$,
\item $J$ gives rise to a \emph{nonnegative definite kernel }in the following
sense 
\begin{equation}
\sum_{i,j=1}^{n}c_{i}c_{j}J(x_{i}-x_{j})\ge0,\label{eq:pos-defi-1}
\end{equation}
for any choice of $n\in\N$, $\{x_{1},\ldots,x_{n}\}\subset\Bcal$,
and $\{c_{1},\ldots,c_{n}\}\subset\R$. 
\end{enumerate}
\end{assumption}

This assumption that the kernel $w$ in \eqref{eq:introAmari} only
depends on the difference is typical for neural field models~\citep{Bressloff}.
Obviously, by (i), (ii) above, $K$ is a self-adjoint Hilbert-Schmidt
operator\footnote{Let $U,V$ be separable Hilbert spaces, $L_{2}(V):=L_{2}(V,V)$, where
$L_{2}(U,V)$ denotes the space of \emph{Hilbert-Schmidt operators}
from $U$ to $V$.} on $H$ that is, $K\in L_{2}(H)$. By a generalization of Mercer's
theorem to compact spaces, $K$ is even of trace class, that is, a
nuclear operator, cf. \citep[Theorem 2.6 and the remark thereafter]{Ferreira:2008jv}.
Denote by $\sigma(K)=\C\setminus\rho(K)$ the \emph{spectrum} of $K$,
where $\rho(K)$ is the \emph{resolvent set}, see e.g. \citep{ReSi1}.
By (iii) above and by \citep[Theorem 2.1]{Ferreira:2009dra} $K$
is nonnegative definite (as a linear operator on $H$), that is, 
\[
\int_{\Bcal}f(x)Kg(x)\,\txtd x\ge0,\quad f,g\in L^{2}(\Bcal),
\]
see also \citep[Theorem 3.1]{Ferreira:2013ei}. See Section \ref{sec:Examples}
below for examples of kernels that satisfy (\ref{eq:pos-defi-1}).
In Section \ref{sec:Examples}, we shall discuss also particular examples
of kernels and models, among others, from works of \citet{neuralfields2014ch1, neuralfields2014ch4, neuralfields2014ch5, neuralfields2014ch9, VeltzFaugeras2010}.

Furthermore, on a filtered normal probability space $(\Omega,\Fcal,\{\Fcal_{t}\}_{t\ge0},\P)$,
let $\{W_{t}\}_{t\ge0}$ be a \emph{cylindrical Wiener process} with
values in $H$ and covariance $Q:=\operatorname{Id}$ equal to the
identity operator on $H$, see Appendix \ref{sec:Cylindrical-Wiener-processes}
for details. Now, let $B\in L(H)$ denote the coefficient/covariance
operator for the noise term in our equation.

\begin{assumption} \label{assu:HS-op} Assume that the covariance
operator of the noise $B\in L_{2}(H)$ is Hilbert-Schmidt, nonnegative
and symmetric. \end{assumption}

This assumption essentially enforces the need for spatial correlations
in dimension $d\geq2$, see \citep{DaPrZa:2nd}. However, these correlations
are indeed quite natural in the modeling setup so we do not give up
relevant generality here. See e.g.~\citep[p. 91, equation (6.55)]{Bressloff}
or \citep[Equation (9.44)]{neuralfields2014ch9} for a justification
in terms of neuroscience to consider a continuum limit model for spatial
correlations of the noise.

Now, we have access to stochastic integration theory in infinite dimensions,
in particular, the expressions 
\[
BW_{t}=\int_{0}^{t}B\,\txtd W_{s},\text{\quad\ensuremath{\text{and}\quad(v,BW_{t})_{H},\;v\in H,\;t\ge0},}
\]
for the cylindrical Wiener process $\{W_{t}\}_{t\ge0}$ as above can
be given a meaning, see \citep[Proposition 4.26 and its proof]{DaPrZa:2nd}.
Note that by our assumptions, $\{BW_{t}\}_{t\ge0}$ is just a $B^{2}$-Wiener
process on $H$. However, we choose to make use of the more complicated
formalism of cylindrical Wiener processes (see Appendix \ref{sec:Cylindrical-Wiener-processes})
for the reason that we would like to consider different assumptions
on regularity for $B$ (see e.g. Assumption \ref{assu:5} below) so
that $B$ helps to keep track of the spatial correlations in the notation.

We continue by recalling the main solution concept.

\begin{defn} \label{def:sln1} A\emph{ solution} to (\ref{eq:spde-nongradient})
is a predictable stochastic process $U\in L^{2}(\Omega\times[0,T];H)$
with $\P$-a.s. Bochner integrable trajectories $t\mapsto U_{t}$
and with\footnote{Or, more generally, with $u_{0}\in L^{2}(\Omega,\Fcal_{0},\P;L^{2}(\Bcal))$.}
$U_{0}=u_{0}\in H=L^{2}(\Bcal)$ such that 
\[
(U_{t},v)_{H}=(u_{0},v)_{H}+\int_{0}^{t}(-\alpha U_{s}+KF(U_{s}),v)_{H}\,\txtd s+\eps\,(v,BW_{t})_{H},
\]
for every $v\in H$, $t\in[0,T]$, $\P$-a.s. \end{defn}

\begin{lem} Under our assumptions, for $u_{0}\in H$, a \emph{mild
solution} in the sense of \citep[Chapter 7]{DaPrZa:2nd} to (\ref{eq:spde-nongradient})
is a solution to (\ref{eq:spde-nongradient}) and vice versa. \end{lem}

\begin{proof} We note that $F:H\to H$ is a bounded operator and
that the domain of the infinitesimal generator of the $C_{0}$-semigroup
$\{e^{-t\alpha}\}_{t\ge0}$ is all of $H$ (which is just identity
times $-\alpha$). The claim follows now from \citep[Theorem 5.4 (and its proof) and Appendix A]{DaPrZa:2nd}.
See also \citep[Proposition G.0.5]{Liu:2015es}. \end{proof}
\begin{rem} Note that in our case, the notion of a solution in $H$
coincides both with the so-called \emph{analytically strong solutions}
and \emph{analytically weak solutions} as defined e.g. in the book
by \citet[Appendix G]{Liu:2015es}. We also note that a solution is
called \emph{weak solution} in the book by \citet{DaPrZa:2nd}. \end{rem}

For $u_{0}\in H$, we may assume that that there exists a unique solution
to~(\ref{eq:spde-nongradient}) and that $U\in C([0,T];H)$, $\P$-a.s.,
see \citep[p.~5]{KuehnRiedler14} and also \citep[Proposition 7.5]{DaPrZa:2nd}.

The next step is to study the gradient structure.

\subsection{\label{subsec:Gradients-in} Gradients in (nonlocal) Hilbert space}

Let us begin with some preliminary considerations. By the spectral
theorem, under Assumption \ref{assu:new}, $K$ has at most countably
many points in its spectrum $\sigma(K)$, all of those being real,
nonnegative and zero being their only accumulation point.

Note that upon setting 
\[
\bar{g}(x):=\left\{ \begin{aligned} & g(x), &  & \quad x\in\Bcal,\\
 & 0, &  & \quad x\in\R^{d}\setminus\Bcal,
\end{aligned}
\right.
\]
we have the convolution representation $Kg=J\ast\bar{g}$. Let $\{\lambda_{i}\}_{i\in\N}$
be the sequence of positive real eigenvalues of $K$ on $H$. In our
situation, w.l.o.g.~$\lambda_{i}\to0$, for $i\to\infty$, $\{\lambda_{i}\}\in\ell^{2}$
and there is a (possibly finite) orthonormal system $\{e_{i}\}_{i\in\N}$
of eigenfunctions for $K$. By the spectral theorem, 
\[
H=\ker K\oplus\overline{\operatorname{span}\{e_{i}\}}
\]
where the decomposition is orthogonal, cf. \citep{ReSi1}. Denote
\[
S:=(\ker K)^{\perp}=\overline{\operatorname{span}\{e_{i}\}}.
\]
$S$ is a closed Hilbert subspace of $H$. The space $S$ is key to
reformulate the problem. The next step is to endow it with a norm.
To emphasize this, we set 
\[
H_{-1}:=S
\]
and consider the (\emph{nonlocal}) norm 
\[
\|g\|_{-1}:=\|g\|_{H_{-1}}:=\|K^{-\frac{{1}}{2}}g\|_{H},\quad g\in S,
\]
where $K^{-\frac{{1}}{2}}:S\to S$ is the operator square root of
the Moore-Penrose pseudoinverse $K^{-1}:S\to S$ of $K:H\to S$, see
\citep[Section 2.1.2]{Hagen:vz}. $H_{-1}$ is a separable Hilbert
space with nonlocal norm and inner product $(\cdot,\cdot)_{-1}:=(K^{-\frac{1}{2}}\cdot,K^{-\frac{1}{2}}\cdot)_{H}$.
Note that $K^{-1}$ is nonnegative on $S$ and that $K^{-\frac{{1}}{2}}$
is the pseudoinverse of $K^{\frac{{1}}{2}}:H\to S$. The space $H_{-1}$
will be used below to reformulate the neural field equation. For later
use, we also note the relation 
\begin{equation}
\|g\|_{H}\le\|K^{\frac{1}{2}}\|_{L(H)}\|g\|_{-1}\quad g\in S,\label{eq:op-bound}
\end{equation}
which just follows from the definitions and where for separable Hilbert
spaces $U,V$, we denote by $\|\cdot\|_{L(U,V)}$ the \emph{operator
norm} and we set $\|\cdot\|_{L(U)}:=\|\cdot\|_{L(U,U)}$.

Let us also define $H_{1}:=K(S)$ with norm $\|g\|_{H_{1}}:=\|K^{\frac{{1}}{2}}g\|_{H}$.
Of course, $\|\cdot\|_{H_{1}}$ is also defined for elements in $H$
or $S$, however, it is zero on $\ker K$ and $\{h\in S\;:\;K^{2}h=0\}$.

At this point, we note that our decomposition of the kernel operator
$K$ has some similarities to the approximation by so-called \emph{Pincherle-Goursat
kernels (PG-kernels)}, see e.g. \citep{tricomi1985integral}, which
serve as finite-dimensional range approximations of $L^{2}$-kernels.
\citet{VeltzFaugeras2010} have used the PG-kernel decomposition for
the analysis of systems of neural field equations. They utilize the
explicit representation of PG-kernels for a reduction of the neural
field model considered by them to a finite number of ODEs via an orthogonal
decomposition of the eigenspaces (compare also Subsection \ref{subsec:Galerkin-approximation}
below). One main difference of their approach to ours is that we require
our kernels to be nonnegative in order to manage the case of possibly
infinitely many distinct eigenvalues while still keeping a Hilbert
space setup.

Let $\vp:\R\to\R$ be any primitive function/antiderivative of $f$.
For example, we may take $\varphi(t):=\int_{0}^{t}f(s)\,\txtd s$,
$t\in\R$. Define
\[
\Phi(u):=\int_{\Bcal}\vp(u(x))\,\txtd x,\quad u\in H,
\]
and 
\[
\Psi(u)=\Psi^{\alpha}(u):=\frac{{\alpha}}{2}\int_{\Bcal}|(K^{-\frac{{1}}{2}}u)(x)|^{2}\,\txtd x=\frac{{\alpha}}{2}\|u\|_{-1}^{2},\quad u\in S.
\]

\begin{lem} \label{lem:energy_L2} $\Phi$ is well-defined, finite
for all $u\in H$ and continuous on $H$. Furthermore, we have that
\[
\txtD\Phi(u)h=(F(u),h)_{H},\quad u,h\in H,
\]
where $\txtD\Phi(u)h$ denotes the G\^{a}teaux-directional derivative
of $\Phi$ in $u$ and in direction $h$. \end{lem}

\begin{proof} We start by proving the claimed properties of $\Phi$.
By the Lipschitz property of $f$, 
\begin{equation}
|f(r)|\le|f(0)|+\operatorname{Lip}(f)|r|,\quad r\in\R.\label{eq:Lip}
\end{equation}
For $r>0$, by the mean-value theorem, there exists $r_{0}\in(0,r)$
with 
\[
\vp(r)-\vp(0)=rf(r_{0}).
\]
Hence 
\[
|\vp(r)|\le\operatorname{Lip}(f)|r||r_{0}|+|f(0)||r|+|\vp(0)|\le\operatorname{Lip}(f)|r|^{2}+|f(0)||r|+|\vp(0)|.
\]
Replacing $r$ by $-r$ yields the same bound for all $r\in\R$. The
first claim, in particular, the continuity of $\Phi$ on $H$ follows
now by an application of the Nemytskii theorem, see e.g. \citep[Theorem II.3.2]{Show},
that is, $\varphi(\cdot):L^{2}(\Bcal)\to L^{1}(\Bcal)$ gives rise
to a continuous nonlinear (Nemytskii) operator.

For the second claim, let $u,h\in H$. Let $t\in\R$, $t\not=0$.
By the mean-value theorem, for almost every $x\in\Bcal$, there exists
$\theta_{x,t}\in(0,1)$ with 
\[
\frac{{1}}{t}\vp(u(x)+th(x))-\vp(u(x))=f(u(x)+\theta_{x,t}th(x))h(x)
\]
and the right-hand side converges to $f(u(x))h(x)$ as $t\to0$ for
almost every $x\in\Bcal$. We have that 
\[
\frac{{1}}{t}[\Phi(u+th)-\Phi(u)]=\frac{{1}}{t}\int_{\Bcal}[\vp(u(x)+th(x))-\vp(u(x))]\,\txtd x.
\]
By the Lipschitz property of $f$ and Lebesgue's dominated convergence
theorem, the right hand side converges to $\int_{\Bcal}f(u(x))h(x)\,\txtd x=(F(u),h)_{H}$
as $t\to0$. \end{proof}

Lemma~\ref{lem:energy_L2} already indicates, how we might be able
deal with the nonlocal term. We define the functional 
\begin{equation}
\Theta(u):=-\Phi\restriction_{S}(u)+\Psi(u),\quad u\in S,\label{eq:theta}
\end{equation}
where $\Phi\restriction_{S}$ denotes the restriction of $\Phi$ to
$S$. $\Theta$ is obviously continuous on $H_{-1}$. Furthermore,
we have the following key lemma. \begin{lem} \label{lem:energy_H-1}We
have that 
\[
\txtD\Phi\restriction_{S}(u)h=(KF(u),h)_{-1},\quad u,h\in H_{-1},
\]
where $\txtD\Phi\restriction_{S}(u)h$ denotes the G\^{a}teaux-directional
derivative of $\Phi\restriction_{S}$ in $u$ and in direction $h$.

Furthermore, we have that 
\[
\txtD\Theta(u)h=\alpha(u,h)_{-1}-(KF(u),h)_{-1},\quad u,h\in H_{-1},
\]
where $\txtD\Theta(u)h$ denotes the G\^{a}teaux-directional derivative
of $\Theta$ in $u$ and in direction $h$.\end{lem}

\begin{proof} Clearly, $u,h\in H_{-1}\hookrightarrow H$ and thus
by the proof of Lemma \ref{lem:energy_L2}, 
\[
\begin{aligned}\txtD\Phi(u)h= & (F(u),h)_{H}=(KF(u),K^{-1}h)_{H}\\
= & (K^{-\frac{{1}}{2}}KF(u),K^{-\frac{1}{2}}h)_{H}=(KF(u),h)_{-1}
\end{aligned}
\]
which proves the first claim.

We get from the book by \citet[p.~91]{Show}, that the Fr\'{e}chet
derivative and thus the G\^{a}teaux-directional derivative of $\Psi$
is given by the following formula $\txtD\Psi(u)h=\alpha(u,h)_{-1}$.
The claim follows now from the preceding discussion, Lemma \ref{lem:energy_L2}
and the above. \end{proof} 

\begin{rem} In the case that $K$ is nonpositive definite, we can
redefine $H_{-1}$ by replacing $K$ by $-K$ in the definition. Now,
by changing the sign for $\Theta$ in (\ref{eq:theta}), we obtain
a gradient by a similar procedure. We can interpret the case of nonnegative
definite symmetric kernels as domination by excitation, while the
case of nonpositive definite kernels corresponds to domination of
the inhibition effects. \end{rem}

\subsection{\label{subsec:Gradient-flow-formulation} Gradient flow formulation}

Having completed the nonlocal functional setting, we now consider
the stochastic nonlocal evolution equation
\begin{equation}
\txtd V_{t}=-\txtD\Theta(V_{t})\,\txtd t+\eps B\,\txtd W_{t},\quad V_{0}=u_{0}\in H_{-1},\label{eq:spde-gradient}
\end{equation}
where $\Theta$ is defined as in \eqref{eq:theta} and where $W$
is a cylindrical Wiener process with values in $H$ and $B\in L_{2}(H)$
is nonnegative and symmetric as in the beginning of this section.
Due to the change of the ambient space, we give a further definition
for a solution to \eqref{eq:spde-gradient} as follows. \begin{defn}\label{def:sln2}
A\emph{ solution} to (\ref{eq:spde-gradient}) is an $\{\Fcal_{t}\}$-adapted
stochastic process $V\in L^{2}(\Omega\times[0,T];H_{-1})$ with $V_{0}=u_{0}\in H_{-1}$
and 
\[
(V_{t},h)_{-1}=(u_{0},h)_{-1}-\int_{0}^{t}\txtD\Theta(V_{s})h\,\txtd s+\eps\,(h,BW_{t})_{-1},
\]
for every $h\in H_{-1}$ and every $t\in[0,T]$. \end{defn}

One reason for introducing several notions of solutions is the aim
of their formal comparison. In order to justify their applicability,
we will prove that the solutions to our gradient-flow formulation
indeed coincide with the solution concept for the original formulation
of the neural field equations. In particular, in order to prove the
invariance of the smaller subspace $H_{-1}$ under the solutions in
the original formulation for initial data in $H_{-1}$, we need to
assume that the spatial correlations of the noise $W$, which are
given by the covariance operator $B$, are sufficiently small relative
to the spectrum of $K$ in a summable way, as formally captured in
the following hypothesis.

\begin{assumption}\label{assu:5} Assume that 
\begin{equation}
B\in L_{2}(H,H_{-1})\quad\text{and}\quad BK^{-1}\in L_{2}(H_{-1},H).\label{eq:smallnoiseassumption}
\end{equation}
\end{assumption}

In the case that $K$ and $B$ commute, the two conditions in \eqref{eq:smallnoiseassumption}
reduce to one, as can be seen as follows.

\begin{ex} \label{ex:jointly-diagonal} Consider the situation that
$B\in L_{2}(H)$ is diagonalized with respect to the same orthonormal
system as $K$ with eigenvalues 
\[
Be_{i}=b_{i}e_{i}\quad i\in\N.
\]
Consider the condition

\begin{equation}
\left\{ b_{i}^{2}\lambda_{i}^{-1}\right\} \in\ell^{1}.\label{eq:spectral_decay}
\end{equation}
Then we claim that \eqref{eq:spectral_decay} holds if and only if
\eqref{eq:smallnoiseassumption} holds.

Indeed, let $B$ have the specific decomposition as assumed. Then
$B$ and $K$ (and thus $B$ and $K^{-1}$) commute on $S$. Then
\[
\sum_{i=1}^{\infty}\|Be_{i}\|_{H_{-1}}^{2}=\sum_{i=1}^{\infty}(K^{-\frac{1}{2}}Be_{i},K^{-\frac{1}{2}}Be_{i})_{H}=\sum_{i=1}^{\infty}(B^{2}K^{-1}e_{i},e_{i})_{H}=\sum_{i=1}^{\infty}\frac{b_{i}^{2}}{\lambda_{i}}
\]
and thus the first part of \eqref{eq:smallnoiseassumption} holds
if and only if \eqref{eq:spectral_decay} holds. For the second part,
note that $\{\sqrt{\lambda_{i}}e_{i}\}_{i\in\N}$ is an orthonormal
basis for $H_{-1}$, so that
\[
\sum_{i=1}^{\infty}\|\sqrt{\lambda_{i}}BK^{-1}e_{i}\|_{H}^{2}=\sum_{i=1}^{\infty}\lambda_{i}\left\Vert \frac{b_{i}}{\lambda_{i}}e_{i}\right\Vert _{H}^{2}=\sum_{i=1}^{\infty}\frac{b_{i}^{2}}{\lambda_{i}}.
\]
Hence \eqref{eq:spectral_decay} holds if and only if \eqref{eq:smallnoiseassumption}
holds.

For instance, $B:=K$ is an example which yields relation \eqref{eq:spectral_decay},
as $K$ is of trace class in $H$ by Mercer's theorem, see also \citep[p. 91, equation (6.55)]{Bressloff}.
\end{ex}

We can compare the solutions to (\ref{eq:spde-nongradient}) and (\ref{eq:spde-gradient})
in the following way. The mentioned invariance result is postponed
to the next subsection. \begin{prop} \label{prop:compare-solutions}Suppose
that $u_{0}\in H_{-1}$. Under Assumptions 1--4, a solution to (\ref{eq:spde-nongradient})
in the sense of Definition \ref{def:sln1} is a solution to (\ref{eq:spde-gradient})
in the sense of Definition \ref{def:sln2}. \end{prop}

\begin{proof} Let $U$ be a solution to~(\ref{eq:spde-nongradient}).
By Theorem \ref{thm:invariance} below, we see that $u_{0}\in H_{-1}$
and the assumed relation \eqref{eq:smallnoiseassumption} imply that
\[
U\in L^{2}(\Omega;L^{\infty}([0,T];H_{-1})).
\]
For any $h\in H_{-1}$ there exists $v\in H$ with $Kv=h$. We see
that 
\[
(U_{t},v)_{H}=(u_{0},v)_{H}+\int_{0}^{t}(-\alpha U_{s}+KF(U_{s}),v)_{H}\,\txtd s+\eps\,(v,BW_{t})_{H},
\]
is equivalent to 
\[\begin{split}
&(U_{t},K^{-1}h)_{H}\\
=&(u_{0},K^{-1}h)_{H}+\int_{0}^{t}(-\alpha U_{s}+KF(U_{s}),K^{-1}h)_{H}\,\txtd s+\eps\,(K^{-1}h,BW_{t})_{H}.
\end{split}\]
Thus, just using the definitions of the nonlocal norms, we get 
\[
(U_{t},h)_{-1}=(u_{0},h)_{-1}+\int_{0}^{t}(-\alpha U_{s}+KF(U_{s}),h)_{-1}\,\txtd s+\eps\,(h,BW_{t})_{-1}.
\]
Finally, the representation of $\txtD\Theta(\cdot)$, proved above,
yields 
\[
(U_{t},h)_{-1}=(u_{0},h)_{-1}-\int_{0}^{t}\txtD\Theta(U_{s})h\,\txtd s+\eps\,(h,BW_{t})_{-1}.
\]
\end{proof} \begin{cor} Suppose that $u_{0}\in H_{-1}$. Assume
that Assumptions 1--4 hold. Then a solution to (\ref{eq:spde-gradient})
in the sense of Definition \ref{def:sln2} is a solution to (\ref{eq:spde-nongradient})
in the sense of Definition \ref{def:sln1}. \end{cor}

\begin{proof} This is proved by reading the chain of equations in
the proof of Proposition \ref{prop:compare-solutions} backwards.
\end{proof} Although we now know that solutions are indeed equivalent,
we have to show that they stay in the nonlocal space $H_{-1}$ if
they start in $H_{-1}$ as our gradient flow formulation is only valid
for this space.

\subsection{\label{subsec:Galerkin-approximation} Galerkin approximation and
invariance of $H_{-1}$ under the flow}

The idea to show invariance is to use a finite-system approximation
idea in combination with It\^{o}'s formula. Consider $u_{0}\in H$
and let $u^{i,N}$, $N\in\N$, $1\le i\le N$ be a solution to the
$N$-dimensional Galerkin system of stochastic ordinary differential
equations 
\begin{equation}
\txtd u_{t}^{i,N}=[-\alpha u_{t}^{i,N}+(KF)^{i,N}(u_{t}^{1,N},\ldots,u_{t}^{N,N})]\,\txtd t+\eps(Be_{i},\txtd W_{t})_{H},\quad u_{0}^{i,N}=u_{0},\label{eq:Galerkin}
\end{equation}
for each $1\le i\le N$. Here, $(KF)^{i,N}$ are given by 
\[\begin{split}
&(KF)^{i,N}(u_{t}^{1,N},\ldots,u_{t}^{N,N})\\
:=&\int_{\Bcal}f\left(\sum_{j=1}^{N}u^{j,N}(x)e_{j}(x)\right)\left(\int_{\Bcal}J(x-y)e_{i}(y)\,\txtd y\right)\,\txtd x,
\end{split}\]
compare with \citep[Section 8]{KuehnRiedler14}. Note that we can
replace each of the noise terms $\eps(Be_{i},\txtd W_{t})_{H}$ by
multiples of independent real-valued Brownian motions $\eps b_{i}\,\txtd\beta_{t}^{i}$
if the cylindrical Wiener process $W$ is modeled with respect to
the orthonormal system $\{e_{i}\}$ and if $B$ and $K$ are codiagonal,
see Example \ref{ex:jointly-diagonal}. As we do not want to assume
this for the sake of generality, we shall accept the constraint that
the modes of the noise may be probabilistically coupled (solutions
to the Galerkin system exist nevertheless). Let $U$ be a solution
to~(\ref{eq:spde-nongradient}). The $N$th Galerkin approximation
$U^{N}$ to $U$ is defined by 
\begin{equation}
U_{t}^{N}:=\sum_{i=1}^{N}u_{t}^{i,N}e_{i}.\label{eq:galerkin-sln}
\end{equation}
The following theorem can be found in \citep[Theorem 8.1]{KuehnRiedler14}.
\begin{thm} \label{thm:L2-convergence} For every $T>0$, it holds
that 
\[
\lim_{N\to\infty}\left[\sup_{t\in[0,T]}\|U_{t}-U_{t}^{N}\|_{H}\right]=0\quad\P\text{-a.s.}
\]
\end{thm}

Having finished our preparations we can now show that if we choose
the initial datum $u_{0}\in H_{-1}$, we remain in the subspace $H_{-1}$.
This is a consequence of the following invariance result. \begin{thm}
\label{thm:invariance} Let $u_{0}\in H_{-1}$. Assume that Assumptions
1--4 hold. Then $U\in L^{2}(\Omega\times[0,T];H_{-1})$ and there
exists a constant $C>0$ such that 
\begin{equation}
\E\left[\sup_{t\in[0,T]}\|U_{t}\|_{-1}^{2}\right]\le C\|u_{0}\|_{-1}^{2}+\eps^{2}C,\label{eq:H-minus-1-bound}
\end{equation}
where $\eps\ge0$ is as in (\ref{eq:spde-nongradient}). \end{thm}

\begin{proof} Note that since $u_{0}\in H_{-1}$, we may assume that
$U^{N}\in C([0,T];H_{-1})$ $\P$-a.s., see the discussion by \citet[Sections 3 and 8]{KuehnRiedler14}.
Also, in this case, 
\[
(KF)^{i,N}(u_{t}^{1,N},\ldots,u_{t}^{N,N})=\lambda_{i}\int_{\Bcal}f\left(\sum_{j=1}^{N}u^{j,N}(x)e_{j}(x)\right)e_{i}(x)\,\txtd x.
\]
We also have to consider the truncated action of the kernel defined
by 
\[
K^{N}v:=\sum_{i=1}^{N}\lambda_{i}(e_{i},v)_{H}e_{i}.
\]
Applying It\^{o}'s formula (see \citep[Theorem 4.32]{DaPrZa:2nd})
for the functional $v\mapsto\|v\|_{-1}^{2}$ yields 
\begin{equation}
\begin{aligned}\|U_{t}^{N}\|_{-1}^{2}= & \|u_{0}\|_{-1}^{2}+2\int_{0}^{t}\left(-\alpha U_{s}^{N}+K^{N}F(U_{s}^{N}),U_{s}^{N}\right)_{-1}\,\txtd s\\
 & +2\eps\sum_{i=1}^{N}\int_{0}^{t}\left(U_{s}^{N},B^{i}e_{i}\,\txtd\beta_{s}^{i}\right)_{-1}+t\eps^{2}\sum_{i=1}^{N}\sum_{k=1}^{N}(Be_{i},e_{k})_{H}^{2}\lambda_{i}^{-1}.
\end{aligned}
\label{eq:solGal}
\end{equation}
Note that $K^{N}$ is a nonnegative definite operator. We see that
$\txtd t\otimes\P$-a.e., by the spectral decomposition of $U^{N}$
and $K^{-1}$ and orthogonality, 
\begin{equation}
\begin{aligned} & \left(-\alpha U^{N}+K^{N}F(U^{N}),U^{N}\right)_{-1}\\
= & -\alpha\|U^{N}\|_{-1}^{2}+\left(\sum_{j=1}^{N}\lambda_{j}\left(e_{j},F(U^{N})\right)_{H}e_{j},K^{-1}U^{N}\right)_{H}\\
= & -\alpha\|U^{N}\|_{-1}^{2}+\left(\sum_{j=1}^{N}\lambda_{j}\left(e_{j},F(U^{N})\right)_{H}e_{j},\sum_{k=1}^{N}\lambda_{k}^{-1}\left(U^{N},e_{k}\right)_{H}e_{k}\right)_{H}\\
= & -\alpha\|U^{N}\|_{-1}^{2}+\left(F(U^{N}),U^{N}\right)_{H}\\
\le & -\alpha\|U^{N}\|_{-1}^{2}+[|f(0)|+\operatorname{Lip}(f)]\|K\|_{L(H)}\|U^{N}\|_{-1}^{2},
\end{aligned}
\label{eq:ex-bound}
\end{equation}
where we have used (\ref{eq:Lip}) and (\ref{eq:op-bound}). To bound
$U_{t}^{N}$ uniformly in $t\in[0,T]$, we are first going to use the Burkholder-Davis-Gundy
inequality, the Young inequality, and Gronwall's lemma to show that
there exist constants $C:=[|f(0)|+\operatorname{Lip}(f)]\|K\|_{L(H)}>0$,
$\delta\in(0,1)$, $\delta<\frac{1}{2T}\txte^{2(\alpha-C)T}$ and
$\kappa(\delta)>0$, such that the following inequality holds, i.e., we estimate~(\ref{eq:solGal}) as
follows 
\begin{eqnarray*}
&&\E\left[\sup_{t\in[0,T]}\|U_{t}^{N}\|_{-1}^{2}\right]\\
& \le & \txte^{2(C-\alpha)T}\|u_{0}\|_{-1}^{2}+\txte^{2(C-\alpha)T}\delta\E\int_{0}^{T}\|U_{s}^{N}\|_{-1}^{2}\,\txtd s\\
 &  & +\txte^{2(C-\alpha)T}T\eps^{2}\kappa(\delta)\sum_{i=1}^{N}\|Be_{i}\|_{H_{-1}}^{2}\\
 &  & +\txte^{2(C-\alpha)T}T\eps^{2}\sum_{i=1}^{N}\sum_{k=1}^{N}(Be_{i},e_{k})_{H}^{2}\lambda_{i}^{-1}\\
 & \le & \txte^{2(C-\alpha)T}\|u_{0}\|_{-1}^{2}+\frac{1}{2}\E\left[\sup_{s\in[0,T]}\|U_{s}^{N}\|_{-1}^{2}\right]\\
 &  & +\txte^{2(C-\alpha)T}T\eps^{2}\kappa(\delta)\|B\|_{L_{2}(H,H_{-1})}^{2}\\
 &  & \qquad+\txte^{2(C-\alpha)T}T\eps^{2}\sum_{i=1}^{N}(Be_{i},Be_{i})_{H}\lambda_{i}^{-1},
\end{eqnarray*}
where we have used Bessel's inequality in the last step. Clearly,
noting that $\{\sqrt{\lambda_{i}}e_{i}\}_{i\in\N}$ is an orthonormal
basis for $H_{-1}$, we get for every $N\in\N$, 
\[
\sum_{i=1}^{N}(Be_{i},Be_{i})_{H}\lambda_{i}^{-1}=\sum_{i=1}^{N}\|\sqrt{{\lambda_{i}}}BK^{-1}e_{i}\|_{H}^{2}\le\|BK^{-1}\|_{L_{2}(H_{-1},H)}^{2},
\]
and hence by~\eqref{eq:smallnoiseassumption}, we get that there
exists another constant $C>0$ with 
\begin{equation}
\E\left[\sup_{t\in[0,T]}\|U_{t}^{N}\|_{-1}^{2}\right]\le C\|u_{0}\|_{-1}^{2}+\eps^{2}C\label{eq:H-minus-1-bound-N}
\end{equation}
for each $N\in\N$ and thus 
\[
\E\left[\sup_{N\in\N}\sup_{t\in[0,T]}\|U_{t}^{N}\|_{-1}^{2}\right]\le C\|u_{0}\|_{-1}^{2}+\eps^{2}C.
\]
Hence $\{U^{N}\}$ has a subsequence $\{U^{N_{k}}\}$ that converges
in the weak$^{\ast}$ sense in the Banach space $L^{2}(\Omega;L^{\infty}([0,T];H_{-1}))$
to some element $\tilde{U}\in L^{2}(\Omega;L^{\infty}([0,T];H_{-1}))$.
Since $H_{-1}$ is continuously embedded into $H$, we see by Theorem
\ref{thm:L2-convergence} that $\tilde{U}=U$ holds $\txtd t\,\otimes\,\P$-a.e..
We conclude the proof by noting that the $L^{2}(\Omega;L^{\infty}([0,T];H_{-1}))$-norm
is lower semi-continuous with respect to weak$^{\ast}$-convergence.
Hence, we can pass on to the limit $N\to\infty$ in (\ref{eq:H-minus-1-bound-N})
and get that 
\[
\E\left[\sup_{t\in[0,T]}\|U_{t}\|_{-1}^{2}\right]\le C\|u_{0}\|_{-1}^{2}+\eps^{2}C
\]
which indeed just means that the solution stays in $H_{-1}$. \end{proof}

\subsection{\label{subsec:Pathwise-regularity} Pathwise regularity of the flow}

In addition to the construction of the flow, it is often helpful,
sometimes even imperative, to have finer control over its regularity.
Let $V\in L^{2}(\Omega\times[0,T];H_{-1})$ be a solution to (\ref{eq:spde-gradient})
for the initial datum $u_{0}\in H_{-1}$. Consider the Doss-Sussmann
transformation 
\[
Y:=V-\eps BW,
\]
which is an established tool for certain classes of stochastic evolution
equations~\citep{CrauelFlandoli}. Then $Y$ satisfies the random
evolution equation 
\begin{equation}
\txtd Y_{t}=-\txtD\Theta(Y_{t}+\eps BW_{t})\,\txtd t,\quad Y_{0}=u_{0},\label{eq:transformed-spde}
\end{equation}
where $\Theta$ is defined as in \eqref{eq:theta}. If $Y$ is a pathwise
solution for (\ref{eq:transformed-spde}), we may transform back by
setting $V(\omega)=Y(\omega)+\eps BW(\omega)$ for each $\omega\in\Omega$.
Note that we may have to drop the assumption that the filtered probability
space is normal in order to obtain a collection of paths which fits
our purposes. We shall use the formulation of (\ref{eq:transformed-spde})
in order to prove additional regularity for the gradient flow $V$.
\begin{prop} Suppose that for fixed $\omega\in\Omega$, $t\mapsto BW_{t}(\omega)$
is \emph{c\`{a}dl\`{a}g}\footnote{That is, \emph{right-continuous with left limits}.}
in $H_{-1}$ and that 
\begin{equation}
BW(\omega)\in L^{2}([0,T];H_{-1})\label{eq:good-noise}
\end{equation}
Let $V$ be a solution to (\ref{eq:spde-gradient}) for the initial
datum $u_{0}\in H_{-1}$. Then the map $t\mapsto V_{t}$ is weakly
continuous in $H_{-1}$ and strongly right-continuous in $H_{-1}$.
\end{prop}

\begin{proof} The strategy of the argument can, e.g., be found in
another setting in the work by~\citet[Proof of Theorem 2.6]{GT11}.
Since we work with a random evolution equation pathwise, we may apply
the chain rule~\citep[Section III.4]{Show} for the functional $v\mapsto\|v\|_{-1}^{2}$
such that for any $\kappa\ge0$, 
\begin{equation}
\begin{aligned} & \|Y_{t}\|_{-1}^{2}\txte^{-\kappa t}\\
= & \|u_{0}\|_{-1}^{2}-2\int_{0}^{t}\txte^{-\kappa s}(\txtD\Theta(Y_{s}+\eps BW_{s}),Y_{s})_{-1}\,\txtd s-\kappa\int_{0}^{t}\txte^{-\kappa s}\|Y_{s}\|_{-1}^{2}\,\txtd s\\
\le & \|u_{0}\|_{-1}^{2}-2\alpha\int_{0}^{t}\txte^{-\kappa s}\|Y_{s}+\eps BW_{s}\|_{-1}^{2}\,\txtd s\\
&+2\alpha\eps^{2}\int_{0}^{t}\txte^{-\kappa s}(Y_{s}+BW_{s},BW_{s})_{-1}\,\txtd s\\
 & \qquad+2[|f(0)|+\operatorname{Lip}(f)]\|K\|_{L(H)}\int_{0}^{t}\txte^{-\kappa s}\|Y_{s}+\eps BW_{s}\|_{-1}\|Y_{s}\|_{-1}\,\txtd s\\
 & \qquad\qquad-\kappa\int_{0}^{t}\txte^{-\kappa s}\|Y_{s}\|_{-1}^{2}\,\txtd s\\
\le & \|u_{0}\|_{-1}^{2}+\frac{\alpha\eps^{2}}{2}\int_{0}^{t}\txte^{-\kappa s}\|BW_{s}\|_{-1}^{2}\,\txtd s\\
 & +C\int_{0}^{t}\txte^{-\kappa s}\|Y_{s}\|_{-1}^{2}\,\txtd s+C\eps^{2}\int_{0}^{t}\txte^{-\kappa s}\|BW_{s}\|_{-1}^{2}\,\txtd s-\kappa\int_{0}^{t}\txte^{-\kappa s}\|Y_{s}\|_{-1}^{2}\,\txtd s;
\end{aligned}
\label{eq:bound-for-transformation}
\end{equation}
compare with (\ref{eq:ex-bound}). Choosing $\kappa=C$ yields 
\[
\|Y_{t}\|_{-1}^{2}\txte^{-Ct}\le\|u_{0}\|_{-1}^{2}+\eps^{2}\left(C+\frac{\alpha}{2}\right)\int_{0}^{t}\txte^{-Cs}\|BW_{s}\|_{-1}^{2}\,\txtd s,
\]
which is finite by (\ref{eq:good-noise}). As, in particular, 
\[
\sup_{t\in[0,T]}\|Y_{t}\|_{-1}^{2}<\infty
\]
and since $t\mapsto Y_{t}=V_{t}-BW_{t}$ is continuous in $H$, we
get by the continuous embedding $H_{-1}\hookrightarrow H$ that $t\mapsto Y_{t}$
is continuous with respect to the weak topology in $H_{-1}$. Note
that by the same calculation as in (\ref{eq:bound-for-transformation}),
we get that 
\[
\|Y_{t}\|_{-1}^{2}\txte^{-Ct}\le\|Y_{s}\|_{-1}^{2}\txte^{-Cs}+\eps^{2}\left(C+\frac{\alpha}{2}\right)\int_{s}^{t}\txte^{-Cr}\|BW_{r}\|_{-1}^{2}\,\txtd r,
\]
for $T\ge t>s\ge0$. Let $t_{n}\in[0,T]$ with $t_{n}\searrow t$.
We find that 
\[
\|Y_{t_{n}}\|_{-1}^{2}\txte^{-Ct_{n}}\le\|Y_{t}\|_{-1}^{2}\txte^{-Ct}+\eps^{2}\left(C+\frac{\alpha}{2}\right)\int_{t}^{t_{n}}\txte^{-Cs}\|BW_{s}\|_{-1}^{2}\,\txtd s
\]
and thus 
\[
\limsup_{n\to\infty}\|Y_{t_{n}}\|_{-1}^{2}\le\|Y_{t}\|_{-1}^{2}.
\]
By weak continuity in $H_{-1}$, we obtain that $Y_{t_{n}}\to Y_{t}$
in $H_{-1}$. Since $t\mapsto BW_{t}$ is right-continuous in $H_{-1}$,
we get that $V=Y+\eps BW$ is right-continuous in $H_{-1}$. \end{proof}

\section{\label{sec:Ergodicity} Ergodicity and unique invariant measures}

We are now going to exploit the gradient structure for the neural
field equation~(\ref{eq:introAmari}). In particular, we are going
to apply classical results, which allow for a direct computation of
the unique invariant measure of the ergodic semigroup of the (stochastic)
flow, see e.g.~\citep{Marcus:1974cw,Marcus:1978kn}. To apply these
results in our setting, let $V^{x}:[0,T]\times\Omega\to H_{-1}$ be
the unique solution to (\ref{eq:spde-gradient}) with initial datum
$V_{0}^{x}=x\in H_{-1}$. Assume $\eps>0$ here, which rules out the
situation of the deterministic PDE. Define the semigroup 
\[
P_{t}^{\eps}G(x):=\E[G(V_{t}^{x})]\quad G\in\Bcal_{b}(H_{-1}),\;x\in H_{-1},\;t\ge0,
\]
where, for a topological space $X$, $\Bcal_{b}(X)$ denotes the space
of bounded Borel measurable maps from $X$ to $\R$. \begin{defn}
We say that $\{P_{t}^{\eps}\}_{t\ge0}$ is symmetric with respect
to a probability measure $\nu$ on $(H_{-1},\Bcal(H_{-1}))$ if\footnote{For a topological space $X$, $\Bcal(X)$ denotes the \emph{Borel
$\sigma$-algebra}.} 
\[
\int_{H_{-1}}G_{1}(x)P_{t}^{\eps}G_{2}(x)\,\nu(\txtd x)=\int_{H_{-1}}G_{2}(x)P_{t}^{\eps}G_{1}(x)\,\nu(\txtd x)\quad\forall G_{1},G_{2}\in\Bcal_{b}(H_{-1})
\]
for every $t\ge0$. \end{defn}

\begin{rem} If $\{P_{t}^{\eps}\}_{t\ge0}$ is symmetric with respect
to some $\nu$, then $\nu$ is an \emph{invariant measure} for $\{P_{t}^{\eps}\}_{t\ge0}$,
that is, 
\[
(P_{t}^{\eps})^{\ast}\nu=\nu\quad\forall t\ge0,
\]
where 
\[
(P_{t}^{\eps})^{\ast}\nu(E):=\int_{H_{-1}}P_{t}^{\eps}1_{E}\,\nu(\txtd x),\quad E\in\Bcal(H_{-1})
\]
is the dual semigroup, see \citep{DPZ2} for details regarding these
standard semigroup constructions. \end{rem}

Note that we can write equation~(\ref{eq:spde-gradient}) also in
the form 
\[
\txtd X_{t}=(AX_{t}+\txtD\Phi(X_{t}))\,\txtd t+\eps B\,\txtd W_{t},\quad X_{0}=x\in H_{-1},
\]
where $A$ corresponds to the linear decay term.

Clearly, $A$ generates a $C_{0}$-contraction semigroup $\{S_{t}\}_{t\ge0}$
on $H_{-1}=(S,\|\cdot\|_{-1})$ which is extendable to a $C_{0}$-contraction
semigroup $\{S_{t}^{0}\}_{t\ge0}$ on the subset $(S,\|\cdot\|_{H})$
of $H$ by the spectral theorem. The infinitesimal generator of $\{S_{t}^{0}\}_{t\ge0}$
is denoted by $A^{0}$. By the definition of the space $H_{-1}$,
$A^{0}$ is a realization of $-\alpha K^{-1}$.

Also note that by our assumptions above, $A^{0}$ is self-adjoint
and on $(S,\|\cdot\|_{H})$ and nonpositive definite. Set 
\[
\Gamma_{\eps}:=\frac{\eps^{2}}{2}(-A^{0})^{-1}=\frac{2\eps^{2}}{\alpha}K,
\]
which is a trace class (i.e., a nuclear) operator. Denote by $\gamma_{\eps}\sim N(0,\Gamma_{\eps})$
the centered Gaussian measure with covariance operator $\Gamma_{\eps}$,
which is concentrated on $S$, see e.g. \citep{B1}. \begin{thm}
Assume that Assumptions 1--4 hold and that $B=K$. Then we get that
the semigroup $\{P_{t}^{\eps}\}_{t\ge0}$ is strongly Markovian\footnote{See \citep[Chapter 9]{DaPrZa:2nd} for the definition of this notion.}
and symmetric with respect to the measure 
\begin{equation}
\mu^{\eps}(\txtd z):=Z_{\eps}^{-1}\exp(2\eps^{-2}\Phi(z))\,\gamma_{\eps}(\txtd z),\label{eq:mu}
\end{equation}
where $Z_{\eps}:=\int_{S}\exp(2\eps^{-2}\Phi(z))\,\gamma_{\eps}(\txtd z)$.
In particular, $\mu^{\eps}$ is an invariant measure for $\{P_{t}^{\eps}\}_{t\ge0}$.\end{thm}

\begin{proof} We have the correct structure for our equation enabling
us to apply \citep[Theorem 2]{Zabczyk:1989jz}. Note that ``Assumption
(H3)'' in \citep{Zabczyk:1989jz} follows from \citep[Theorem 7.14]{DaPrZa:2nd},
combined with \citep[Theorem 9.21]{DaPrZa:2nd}. We refer also to
\citep{DaPratoZabczyk1988} for an historical account. Compare also
with \citep[Theorem 5]{Mueck1995}. \end{proof} Of course, we
recognize the invariant measure~(\ref{eq:mu}) as the usual Gibbs
measure well-known from statistical physics, yet here it simply ``lives''
on a nonlocal space. However, its biophysical relevance is the same
as the classical one as it can be interpreted as the stationary probability
of a state. In fact, even more can be said about $\mu^{\eps}$. Under
the assumptions above, the measure $\mu^{\eps}$ is in fact unique
and the semigroup $\{P_{t}^{\eps}\}_{t\ge0}$ is ergodic. Recall that
$\delta_{x}(B):=1_{B}(x)$, $B\in\Bcal(H_{-1})$ denotes the \emph{Dirac
measure }and $\|\cdot\|_{\textup{TV}}$ denotes the \emph{total variation}
of a measure. \begin{thm} Assume that Assumptions 1--4 hold and
that $B=K$. Then the measure $\mu^{\eps}$, as defined in (\ref{eq:mu}),
is unique and the semigroup $\{P_{t}^{\eps}\}_{t\ge0}$ is ergodic
and \emph{strong Feller} in the following \emph{restricted sense}:
\[
\|(P_{t}^{\eps})^{\ast}\delta_{x_{n}}-(P_{t}^{\eps})^{\ast}\delta_{x_{0}}\|_{\textup{TV}}\to0
\]
as $n\to\infty$ for any sequence of points $x_{n}\in H_{-1}$, $n\in\N$,
$x_{0}\in H_{-1}$ with $\|x_{n}-x_{0}\|_{-1}\to0$ as $n\to\infty$
and for any $t>0$. \end{thm}

\begin{proof} See \citep[Theorem 2.1]{Maslowski:1989dp}, where ``Assumption
(H4)'' follows in \citep{Maslowski:1989dp} follows from \citep[Theorem 7.14]{DaPrZa:2nd},
combined with \citep[Theorem 9.21]{DaPrZa:2nd}. \end{proof}
Uniqueness of the invariant measure (\ref{eq:mu}) then follows from
its so-called \emph{asymptotic strong Feller property}, see \citep{Hairer:2006ib,Hairer:2011je}
for this notion. \begin{cor} Assume that Assumptions 1--4 hold and
that $B=K$. Then $\{P_{t}^{\eps}\}_{t\ge0}$ is asymptotic strong
Feller. \end{cor}

\begin{proof} See \citep[Remark 3.9]{Hairer:2006ib}. \end{proof}

\begin{rem}

The requirement that $B=K$ is not necessary for the existence of
invariant measures. We have assumed this condition in order to get
an explicit and simple representation for the invariant measure. In
a finite dimensional context, this observation goes back to \citet{Kolmogorov}.
In fact, in order to obtain a symmetrizing Gibbs-type representation
for the invariant measure as in \eqref{eq:mu}, it is necessary and
sufficient to require $B=K$, cf. the exposition by \citet{Mueck1995}.
By an application Krylov-Bogoliubov's method, we note that under Assumptions
1--4, there always exists an invariant measure, see e.g. \citep{ESSARHIR2008,ESSTvG}.

\end{rem}

Under additional monotonicity and weak dissipativity assumptions on
$f$ one might even obtain decay estimates on invariant measures as
e.g.~in the works by \citet{BDP06,LiuToe1}. However, these steps
are beyond the basic framework we develop here.

\section{Examples of nonnegative definite kernels}

\label{sec:Examples}

To illustrate the setting, we also have to provide several concrete
examples. The key restriction that is required for the gradient structure
setup is the nonnegative definiteness assumption for the kernel. Let
us recall the following useful generating function characterization
of positive definite kernels determined by a function $J$ as in Assumption
\ref{assu:new}. \begin{thm} A kernel of the form $(x,y)\mapsto J(x-y)$,
$x,y\in\R^{d}$, for some function $J:\R^{d}\to\R$, is nonnegative
definite in the sense of \eqref{eq:pos-defi-1} if and only if there
exists a finite nonnegative Borel measure $\hat{\sigma}$ on $(\R^{d},\Bcal(\R^{d}))$
for which 
\begin{equation}
J(x)=\int_{\R^{d}}\txte^{\txti\langle y,x\rangle}\,\hat{\sigma}(\txtd y),\quad x\in\R^{d}.\label{eq:fourier}
\end{equation}
\end{thm}

\begin{proof} See \citep[Theorem 7.5]{Ferreira:2013ei} and also
\citep{Sasvari:ek,Stewart:1976ho}. \end{proof} This result
can be applied under the caveat that it essentially holds for $\R^{d}$,
but we need to truncate and rescale accordingly (while restricting
to radially symmetric support) in order to adapt it to use in our
context.

We shall give a couple of classical examples. \begin{ex}\label{ex:kernel-examples}
The following examples for $J$ satisfy relation (\ref{eq:fourier})
for some $\hat{\sigma}$, see e.g. \citep{Sasvari:ek}. 
\begin{enumerate}
\item (\emph{centered Gauss distribution}) $J(x)=\txte^{-\frac{1}{2}\langle x,Mx\rangle}$,
where $M\in\R^{d\times d}$ is a symmetric and positive semi-definite
matrix. 
\item (\emph{centered Cauchy distribution}) $J(x)=\txte^{-\sqrt{\langle x,Mx\rangle}}$,
where $M\in\R^{d\times d}$ is a symmetric and positive semi-definite
matrix. 
\item (\emph{centered Laplace distribution}) $J(x)=\left(1+\frac{1}{2}\langle x,Mx\rangle\right)^{-1}$,
where $M\in\R^{d\times d}$ is a symmetric and positive semi-definite
matrix. 
\item (\emph{uniform distribution on $[-1,1]^{d}$}) $J(x)=J(x_{1},\ldots,,x_{d})=\prod_{j=1}^{d}\frac{\sin(x_{j})}{x_{j}}$,
where the factors of the product are (by definition) equal to $1$
if $x_{j}=0$. 
\item (\emph{symmetric sums of Dirac distributions}) $J(x)=\sum_{i=1}^{\infty}a_{i}\cos(\langle m_{i},x\rangle)$,
where $a_{i}\ge0$ with $\sum_{i=1}^{\infty}a_{i}=1$ and $m_{i}\in\R^{d}$,
$m_{i}\not=\pm m_{j}$ for $i\not=j$.
\end{enumerate}
\end{ex}

In fact, by \citep[Theorem 1.3.13]{Sasvari:ek}, relation (\ref{eq:fourier})
can be verified for some $J$ with the required properties for any
probability measure $\hat{\sigma}$ on $(\R^{d},\Bcal(\R^{d}))$ which
is \emph{symmetric}, that is, $\hat{\sigma}(B)=\hat{\sigma}(-B)$
for all $B\in\Bcal(\R^{d})$. These are precisely the distributions
of random vectors $X$ for which it holds that $X$ and $-X$ have
the same distribution. In this case, we have that 
\[
J(x)=\int_{\R^{d}}\cos(\langle y,x\rangle)\,\hat{\sigma}(\txtd y),\quad x\in\R^{d},
\]
see \citep[Theorem 1.3.13]{Sasvari:ek} again. Hence,
we can always use the Fourier relation (\ref{eq:fourier}) to (explicitly
or numerically) check, whether our gradient-structure formulation
applies. The characterization
extends of course to positive multiples of $\hat{\sigma}$.

For instance, by Example \ref{ex:kernel-examples} (ii), we see that
the exponential weight used e.g. by \citet[Equation (9.49)]{neuralfields2014ch9}
satisfies our assumptions. Also, by Example \ref{ex:kernel-examples}
(v), we see that a finite sum of cosine functions, as considered e.g.
by \citet[Section 5]{VeltzFaugeras2010} for the so-called periodic
ring model introduced by \citet{ShrikiHanselSompolinsky2003}, satisfies
our requirements. The cosine weight with period matching the domain
size can also be found in \citep{neuralfields2014ch4}, we note, however,
that the (non-Lipschitz) Heaviside activation function used therein
is out of the scope of our paper. 

\begin{rem}
\begin{enumerate}
\item We also point out in this context that there seems to be potential
for confusion in the literature regarding nonnegative definiteness
of certain classes of kernels commonly used in neural fields. For
example, consider the (one-dimensional) \emph{Mexican hat kernel}
\begin{equation}
J(x)=\left(1-x^{2}\right)\exp\left(-\frac{x^{2}}{2}\right),\label{eq:MexHat}
\end{equation}
which can easily be normalized by a positive pre-factor, scaled in
the spatial variable and/or extended into higher dimensions. In a
work related to neuroscience one finds the statement that ``the Mexican
hat kernel differs from the other kernels, since it is not positive
definite''~\citep{BachmairSchoell}. While in certain works on machine
learning one finds that ``Mexican hat kernel(s) are Mercer kernel(s)''~\citep{Xieetal},
which just implicitly means that the Mexican hat kernel is nonnegative
definite (or positive semi-definite). Indeed, the second statement
seems correct in view of~(\ref{eq:fourier}) as shown by the simple
calculation 
\begin{equation}
J(x)=\left(1-x^{2}\right)\exp\left(-\frac{x^{2}}{2}\right)=\frac{1}{\sqrt{2\pi}}\int_{\R}\txte^{\txti\xi x}\exp\left(-\frac{\xi^{2}}{2}\right)\xi^{2}~\txtd\xi
\end{equation}
so that Fourier inversion is possible, i.e., there even exists a measure
$\hat{\sigma}$ with an explicit nonnegative density.
\item Let us also discuss another family of kernels, prominent in the literature
of neural fields, with a similar shape than that of \eqref{eq:MexHat},
which are therefore also known as \emph{Mexican hat kernels}. Let
$0<A<1$, $s>1$ and consider
\begin{equation}
J(x)=\exp\left(-\frac{x^{2}}{2}\right)-A\exp\left(-\frac{x^{2}}{s^{2}}\right).\label{eq:Mexican_hat_2}
\end{equation}
We get that
\begin{equation}
J(x)=\frac{1}{\sqrt{2\pi}}\int_{\R}\txte^{\txti\xi x}\left[\exp\left(-\frac{\xi^{2}}{2}\right)-\frac{As}{\sqrt{2}}\exp\left(-\frac{s^{2}\xi^{2}}{4}\right)\right]~\txtd\xi.\label{eq:As-formula}
\end{equation}
In order that $J$ is given by some nonnegative Borel measure as in
\eqref{eq:fourier}, we require that its density, as given inside
the square brackets in \eqref{eq:As-formula}, is nonnegative, which,
by an elementary computation, is true if and only if $\sqrt{2}\le s\le\frac{\sqrt{2}}{A}$.
\item In fact, \citet[Equation (1.29)]{neuralfields2014ch1} give yet another
example of a kernel representing short range excitation and long-range
inhibition (that is, a \emph{Mexican hat kernel}). Namely, for $0<\Gamma<1$,
$\gamma_{1}>\gamma_{2}>0$,
\[
J(x)=\exp(-\gamma_{1}|x|)-\Gamma\exp(-\gamma_{2}|x|).
\]
We have that
\begin{equation}
J(x)=\frac{1}{\sqrt{2\pi}}\int_{\R}\txte^{\txti\xi x}\left[2\left(\frac{\gamma_{1}}{\gamma_{1}^{2}+\xi^{2}}-\Gamma\frac{\gamma_{2}}{\gamma_{2}^{2}+\xi^{2}}\right)\right]~\txtd\xi,\label{eq:mexican_hat3}
\end{equation}
see also \citep[Equation (1.30)]{neuralfields2014ch1}. It can be
seen easily, that the density inside the square brackets in \eqref{eq:mexican_hat3}
is nonnegative if and only if $\Gamma\le\frac{\gamma_{2}}{\gamma_{1}}$.
\item \citet[Equation (5.7)]{neuralfields2014ch5} suggests the following
kernels (among others), for $b>0$,
\[
J(x)=\txte^{-b|x|}\left(b\sin(|x|)+\cos(x)\right)
\]
and the \emph{wizard hat} (see \citep[p. 157]{neuralfields2014ch5})
\[
J(x)=\frac{1}{4}\left(1-|x|\right)\txte^{-|x|}.
\]
As discussed already by \citet{neuralfields2014ch5}, both examples
have a nonnegative Fourier transform and can hence be considered within
the scope of our this work.
\end{enumerate}
\end{rem}

\section{\label{sec:Summary-Discussion} Summary \& Discussion}

In this work we have proved that the stochastic Amari neural field
model has a gradient flow structure for certain types of neural connectivity
kernels. We have used this structure to show well-posedness of the
model in a nonlocal Hilbert space build from the connectivity kernel.
Furthermore, we have shown the existence of unique Gibbs-type invariant
measures for the associated Feller semigroup. These results provide
a strong indication that neural field models can be analyzed using
gradient flow techniques if one builds the actual gradient flow space
using the neural connectivity pattern via a kernel. Yet, for certain
kernels, we conjecture that there is actually no gradient flow structure,
even in an adapted nonlocal Hilbert space. In fact, it seems plausible
to expect this from a neuroscience perspective as we also have to
allow for the possibility of time-periodic patterns to model several
dynamical effects in the brain. From the viewpoint of neurological
disorders as well as multi-stable visual perception effects, our results
also contribute to a very natural idea: if measurements, such as EEG
measurements for epileptic seizures, show a completely different high-oscillation/synchrony
regime in comparison to normal brain functioning, one may expect that
any underlying mathematical model might also undergo a major transition
between regimes. In fact, this transition may not only manifest itself
in changing model parameters but it may find a clear expression in
the mathematical type of the equations themselves. In our work, we
have shown that such a structural model transition indeed is possible.
For certain parameters, the connectivity kernel may satisfy our assumptions
so that typical gradient-flow dynamics is observed leading to energy
dissipation and equilibration towards an invariant measure. Yet, for
other parameters, our kernel may not lead to a gradient flow and this
opens up the possibility of wide variety of neural activity patterns.

\appendix

\section{\label{sec:Cylindrical-Wiener-processes} Cylindrical Wiener processes
in Hilbert spaces}

Let $\{v_{i}\}_{i\in\N}\subset H$ be a complete orthonormal system
for $H$ and let $\{\beta_{t}^{i}\}_{i\in\N}$ be a collection of
independent real-valued standard Brownian motions modeled on a filtered
normal probability space $(\Omega,\Fcal,\{\Fcal_{t}\}_{t\ge0},\P)$.
Then the cylindrical Wiener process $\{W_{t}\}_{t\ge0}$ with covariance
$Q=\operatorname{Id}$ has the formal representation 
\begin{equation}
W_{t}=\sum_{i=1}^{\infty}v_{i}\beta_{t}^{i},\quad t\ge0,\label{eq:wiener-1}
\end{equation}
which is a standard Wiener process in a weaker separable Hilbert space
$U$, such that there exists a Hilbert-Schmidt embedding $\iota:H\to U$,
$\iota\in L_{2}(H,U)$ and $\{W_{t}\}_{t\ge0}$ has the covariance
operator $\iota\iota^{\ast}$. By \citep[Proposition 4.7 and Proposition 4.8]{DaPrZa:2nd},
then $(\iota\iota^{\ast})^{\frac{1}{2}}(U)=H$, and one can always
find $U$ and $\iota$ with the above properties such that the representation
(\ref{eq:wiener-1}) holds; see \citep[Chapter 4]{DaPrZa:2nd} for
details. For our purposes, it is sufficient to set $U$ equal to the
abstract completion of $L^{2}(\Bcal)$ with respect to the alternative
scalar product $(u,v)_{U}:=\sum_{n=1}^{\infty}n^{-2}\hat{u}_{n}\hat{v}_{n}$,
where $u,v\in H$ and $\{\hat{u}_{n}\},\{\hat{v}_{n}\}\in\ell^{2}$
are such that $u=\sum_{n=1}^{\infty}\hat{u}_{n}v_{n}$ and $v=\sum_{n=1}^{\infty}\hat{v}_{n}v_{n}$.
Then $\iota:=\operatorname{Id}$ is a Hilbert-Schmidt embedding from
$H$ into $U$.

\def\cprime{$'$}

\end{document}